\documentclass[12pt]{amsart}
\usepackage{enumerate}
\usepackage[hyperindex,breaklinks,colorlinks,citecolor=blue,pagebackref]{hyperref}
\usepackage[latin1]{inputenc}
\usepackage{amsmath, amsthm, amsfonts, amssymb}
\usepackage{mathrsfs}
\usepackage{graphicx}
\usepackage{latexsym}
\usepackage{fullpage}
\usepackage[all]{xy}
\usepackage{color}
\usepackage{stmaryrd}
\usepackage[mathscr]{euscript}

\newtheorem{theorem}{Theorem}[section]

\newtheorem{proposition}[theorem]{Proposition}
\newtheorem{corollary}[theorem]{Corollary}
\newtheorem{question}[theorem]{Question}

\theoremstyle{remark}

\newtheorem{rmks}{Remarks}[section]
\newtheorem*{Examples}{Examples}

\newcommand{\B}{\mathcal{B}}

\renewcommand{\P}{\mathcal{P}}

\newcommand{\U}{\mathcal{U}}

\newcommand{\MLR}{\mathsf{MLR}}

\newcommand{\dom}{\text{dom}}

\newcommand{\KA}{\mathit{KA}}

\newcommand{\llb}{\llbracket}
\newcommand{\rrb}{\rrbracket}

\definecolor{purple}{rgb}{.9,0.2,.9}

\newcommand{\eqdef}{\mathrel{\mathop:}=}

\newcommand{\cs}{2^\omega}

\newcommand{\uh}{{\upharpoonright}}
\renewcommand{\phi}{\varphi}
\newcommand{\str}{2^{<\omega}}

\usepackage{xcolor}	
	\usepackage{soul}

\title{Effective Aspects of  Bernoulli Randomness}
\author{Christopher P.\ Porter}
\address{Department of Mathematics and Computer Science\\
Drake University\\
Des Moines, IA 50311\\ 
USA}
\email{christopher.porter@drake.edu}

\date{} 

\begin{document}

\begin{abstract}
In this paper, we study Bernoulli random sequences, i.e., sequences that are Martin-L\"of random with respect to a Bernoulli measure $\mu_p$ for some $p\in[0,1]$, where we allow for the possibility that $p$ is noncomputable.  We focus in particular on the case in which the underlying Bernoulli parameter $p$ is proper (that is, Martin-L\"of random with respect to some computable measure).  We show for every Bernoulli parameter $p$, if there is a sequence that is both proper and Martin-L\"of random with respect to $\mu_p$, then $p$ itself must be proper, and explore further consequences of this result.  We also study the Turing degrees of Bernoulli random sequences, showing, for instance, that the Turing degrees containing a Bernoulli random sequence do not coincide with the Turing degrees containing a Martin-L\"of random sequence.  Lastly, we consider several possible approaches to characterizing blind Bernoulli randomness, where the corresponding Martin-L\"of tests do not have access to the Bernoulli parameter $p$, and show that these fail to characterize blind Bernoulli randomness.
\end{abstract}

\maketitle

\section{Introduction}

Algorithmic randomness with respect to biased probability measures on the Cantor space $\cs$ has been studied in two separate strands:  with respect to computable measures, such as in \cite{BieMer09}, \cite{BiePor12}, and \cite{HolPor17}, and with respect to non-computable measures, such as in \cite{DayMil13} and \cite{ReiSla18} (see the recent survey \cite{Por19} for an overview of both approaches).  In this article, we study the interaction of these two strands in the context of Bernoulli measures on $\cs$.  Recall that, for $p\in[0,1]$, the Bernoulli $p$-measure $\mu_p$ is defined by $\mu_p(\llb\sigma\rrb)=p^{\#_0(\sigma)}(1-p)^{\#_1(\sigma)}$
for each basic open subset $\llb\sigma\rrb$ of $\cs$, where, for $i\in\{0,1\}$, $\#_i(\sigma)$ is the number of occurrences of the symbol $i$ in $\sigma$. We refer to such a $p$ as a \emph{Bernoulli parameter}.  We will refer to Martin-L\"of randomness with respect to some Bernoulli measure as \emph{Bernoulli randomness.}

The most significant work on the topic of Bernoulli randomness is due to Kjos-Hanssen \cite{Kjo10}, who focuses in particular on randomness with respect to non-computable Bernoulli measures.  On the standard approach to defining Martin-L\"of randomness with respect to a non-computable measure, one includes the measure as an oracle to be used in enumerating the corresponding Martin-L\"of tests (or, more precisely, a sequence that encodes the values of the measure on basic open sets in some effective way).  In the case of defining Bernoulli randomness with respect to a non-computable Bernoulli parameter $p$, it suffices to use $p$ as an oracle in the definition of the corresponding Martin-L\"of tests.

 By contrast with the standard approach, Kjos-Hanssen considers \emph{blind} Bernoulli randomness, i.e.\ Bernoulli randomness defined in terms of Martin-L\"of tests that do not have access to the Bernoulli parameter $p$.  His two main findings are the following:  First, he shows that for $p\in[0,1]$, every Bernoulli $p$-random sequence computes the Bernoulli parameter $p$.  Second, using this first result, he shows that the standard notion of Bernoulli $p$-randomness (given in terms of tests that have access to the oracle $p$) coincides with the notion of blind Bernoulli $p$-randomness. 
 
In this study, we extend Kjos-Hanssen's investigations in several respects.  First, in Section \ref{sec-3} we study the behavior of sequences that are Bernoulli random with respect to an algorithmically random Bernoulli parameter $p$.  It follows from work of V'yugin \cite{Vyu12} (building upon work in \cite{VovVyu93}) and, independently, Hoyrup \cite{Hoy13}, that if a sequence $x$ is Bernoulli random with respect to a Martin-L\"of random Bernoulli random parameter $p$, that $x$ is \emph{proper}, that is, random with respect to a computable measure on $\cs$.  Here, by a careful analysis of Kjos-Hanssen's first result discussed above, we prove a partial converse of this result:  If a proper sequence $x$ is Bernoulli random with respect to some parameter $p$, then $p$ itself must be proper as well.  We then explore a number of consequences of this theorem.

Second, by our first main result, every Bernoulli random sequence that is proper has Martin-L\"of random Turing degree (i.e., is Turing equivalent to a Martin-L\"of random sequence).  This raises a number of questions about the Turing degrees of Bernoulli random sequences, which we consider in Section \ref{sec-4}.  For instance, we show that for every parameter $p$, there is a Bernoulli $p$-random sequence that has Martin-L\"of random Turing degree, and by contrast, there is some parameter $p$ and some Bernoulli $p$-random sequence that does not have Martin-L\"of random Turing degree.

Lastly, in Section \ref{sec-5} we study several candidate definitions of blind Bernoulli randomness given in terms of initial segment complexity and supermartingales.  It is well-known that a sequence is Bernoulli $p$-random if and only if $\mathit{KA}^p(x\uh n)\geq -\log\mu_p(X\uh n)-O(1)$, where $\mathit{KA}(\sigma)$ is the a priori complexity of $\sigma\in\str$ (defined below in Section \ref{sec-5}).  Here we show that the weaker condition $\mathit{KA}(x\uh n)\geq -\log\mu_p(x\uh n)-O(1)$ does not imply that $x$ is Bernoulli $p$-random, which shows that this condition does not provide a notion of blind randomness in terms of a priori complexity that is equivalent to Kjos-Hanssen's original definition of blind randomness.  From this result, we can easily derive the conclusion that (1) a similar condition in terms of prefix-free Kolmogorov complexity and (2) a notion of randomness defined in terms  of blind supermartingales both fail to characterize Bernoulli randomness.


Before establishing the above-described results, in the following section we provide the requisite background.

\section{Background}\label{sec-2}

\subsection{Notation and Measures}
The set of finite binary strings is denoted $\str$. The space of all
infinite binary sequences is denoted $\cs$ and comes equipped with the
product topology generated by the clopen sets $  \llb\sigma\rrb= \{x \in \cs: \sigma\prec x\}$, where $\sigma \in \str$ and $\sigma \prec x$ means $\sigma$ is an
initial segment of $x$.

A \emph{(probability) measure} $\mu$ on $\cs$ is a function that
assigns to each Borel subset of $\cs$ a number in the unit interval
$[0,1]$ and satisfies
$\mu(\bigcup_{i \in \omega} \B_{i})= \sum_{i \in \omega} \mu(\B_{i})$
whenever the $\B_{i}$'s are pairwise disjoint Borel subsets of $\cs$.  Carath\'{e}odory's
extension theorem guarantees that the conditions
\begin{itemize}
  \item $\mu(\cs)= 1$ and
  \item
  $\mu(\llb\sigma\rrb) = \mu(\llb\sigma0\rrb) + \mu(\llb\sigma1\rrb)$
  for all $\sigma \in \str$
\end{itemize}
uniquely determine a measure on $\cs$\!.  We thus identify a measure
with a function ${\mu\colon \str \to [0,1]}$ satisfying the above
conditions and $\mu(\sigma)$ is often written instead of
$\mu(\llb\sigma\rrb)$.  The Lebesgue measure $\lambda$ on $\cs$ is
defined by $\lambda(\sigma) = 2^{-|\sigma|}$ for each string
$\sigma\in \str$, where $|\sigma|$ denotes the length of $\sigma$. The space of all measures on $2^{\omega}$ is denoted $\mathscr{P}(\cs)$.  

\subsection{Some computability theory}
\label{sec:some-comp-theory}

We assume the reader is familiar with the basic concepts of
computability theory as found, for instance, in the early chapters of
\cite{Nie09} or \cite{DowHir10}.  We review a few useful concepts.

A $\Sigma^0_1$ \emph{class} $S\subseteq \cs$ is an effectively open
set, i.e., an effective union of basic clopen subsets of $\cs$\! and a
$\Pi^{0}_{1}$ \emph{class} is the compliment of a $\Sigma^{0}_{1}$
class.

A map $\Phi\colon \subseteq \cs \to \cs$ is a Turing functional if there is an oracle Turing machine that when given
$x \in \dom(\Phi)$ (as an oracle) and $k \in \omega$ outputs
$\Phi(x)(k)$ (unless $\Phi(x)(k)$ is undefined). Relativization of this notion to some $z\in\cs$ leads to the notion of a $z$-computable
function.

A measure $\mu$ on $\cs$ is computable if $\mu(\sigma)$ is a
computable real number, uniformly in $\sigma\in 2^{<\omega}$\!.  Note that the Bernoulli $p$-measure $\mu_p$ (as defined in the introduction)
is computable if and only if the Bernoulli parameter $p\in[0,1]$ is computable.
If
$\mu$ is a computable measure on $\cs$ and
$\Phi\colon \subseteq \cs \to \cs$ is a Turing functional defined on
a set of $\mu$-measure one, then the \emph{pushforward measure}
$\mu_\Phi$ defined by
\begin{displaymath}
  \mu_\Phi(\sigma)=\mu(\Phi^{-1}(\sigma))
\end{displaymath}
for each $\sigma\in \str$ is a computable measure.

%
%
%

\subsection{Martin-L\"of randomness with respect to various measures}
 Recall that for a fixed computable measure $\mu$ on $\cs$\! and $z\in \cs$\!, a 
 \emph{$\mu$-Martin-L\"of test relative to $z$} (or simply a \emph{$\mu$-test relative to $z$}) is a uniformly
    $\Sigma^{0,z}_{1}$ sequence $(\U_{i})_{i\in\omega}$ of subsets of
    $\cs$ with $\mu (\U_{n}) \leq 2^{-i}$\! for every $i\in\omega$. $x\in \cs$ passes such a test $(\U_{i})_{i\in\omega}$
    if $x\notin \bigcap_{i\in\omega} \U_{i}$ and $x$ is $\mu$-Martin-L\"of random relative to
    $z$ if $x$ passes every $\mu$-Martin-L\"of test relative to
    $z$. The set of all such $x$'s is denoted by $\MLR_{\mu}^{z}$.  As in the introduction, we say that $x\in\cs$ is proper if $x$ is Martin-L\"of random with respect to some computable measure.
Lastly, for each choice of $\mu$ and $z$ as above, there is a single, \emph{universal}, $\mu$-test relative to $z$,
$(\U_{i})_{i\in \omega}$ such that $x \in \MLR_{\mu}^{z}$ if and only if $x$ passes $(\U_{i})_{i\in \omega}$. 


For $p\in[0,1]$, a sequence $x\in\cs$ is \emph{Bernoulli $p$-random} if it passes every $\mu_p$-Martin-L\"of test relative to $p$.  The collection
of Bernoulli $p$-random sequences will be written as $\MLR_{\mu_p}$ (we suppress the oracle $p$).
We say that $x\in\cs$ is a \emph{Bernoulli random} if it is Bernoulli $p$-random for some $p\in[0,1]$.

In the next section, we will briefly consider Martin-L\"of random members of  $\mathscr{P}(\cs)$, which can be defined by defining the notion of a Martin-L\"of test in the setting of $\mathscr{P}(\cs)$ (or more generally, any computable metric space, as in, for instance, \cite{Hoy13}), or by representing each measure by a  sequence and defining a random measure to be one that has a random representing sequence (as in \cite{Cul15}).  The former is a straightforward extension of the definition of Martin-L\"of randomness for computable measures on $\cs$; as we will not formally define such tests in the sequel, we omit the details.

We conclude this section with a pair of useful results concerning the interaction between Turing functionals and Martin-L\"of randomness:

\begin{theorem}\label{thm-pres}
  Let $\Phi\colon \subseteq \cs \to \cs$ be a Turing functional and $\mu\in\mathscr{P}(\cs)$ be computable with
  $\mu(\dom(\Phi))=1$.
  \begin{itemize}
    \item[(i)] (Preservation of Randomness \cite{ZvoLev70}) If $x \in \MLR_{\mu}$ then
    $\Phi(x) \in \MLR_{\mu_\Phi}$.
    \item[(ii)] (No
  Randomness Ex Nihilo \cite{She86}) If $y \in \MLR_{\mu_\Phi}$, then there is some
    $x\in \MLR_{\mu}$ such that $\Phi(x)=y$.
  \end{itemize}
  \label{thm:PoR-and-NReN}
\end{theorem}

\section{Random Bernoulli measures}\label{sec-3}

\subsection{Mixtures of Bernoulli measures}  The mixture of Bernoulli measures was first studied by De Finetti \cite{DeF31}.
Given a measure $P$ on $\mathscr{P}(\cs)$, the \emph{barycenter} of
$P$ is the measure $\xi_P$ on $\cs$ defined by
\begin{displaymath}
  \xi_P(\mathcal{U}) = \int \mu(\mathcal{U})\, dP(\mu)
\end{displaymath}
for all Borel $\U\subseteq\cs$.  In the case that $P$ is concentrated on a set of Bernoulli measures in $\mathscr{P}(\cs)$, we say that the barycenter $\xi_P$ is a \emph{mixture of Bernoulli measures}.

A measure $\xi$ is \emph{exchangeable} if the $\xi$-probability of a
string $\sigma$ being an initial segment of $x\in\cs$ is the same as the
$\xi$-probability of $\sigma$ occurring as a subword at any fixed block of bits of length $|\sigma|$ in $x$. A
classical result in probability theory is De Finetti's theorem, which
says that $\xi$ is exchangeable if and only if it is the mixture of Bernoulli
measures. 
Freer and Roy \cite{FreRoy12} proved that in this setting $\xi$
is computable if and only if $P$ is computable. Hoyrup then
generalized this result via the following theorem.

\begin{theorem}[\cite{Hoy13}]\label{thm-barycenter}
  If $P$ is a computable measure on $\mathscr{P}(\cs)$, then its
  barycenter measure $\xi_P$ is computable and
  \begin{displaymath}
    \MLR_{\xi_P}=\bigcup_{\mu\in\MLR_P}\MLR_\mu.
  \end{displaymath}
\end{theorem}

This result provides a useful tool for studying random Bernoulli measures.  Given a computable measure $\nu$ on $[0,1]$, $\nu$ induces a measure $P_\nu$ on $\mathscr{P}(\cs)$ that satisfies 
\[
\int\mu(\U)dP_\nu(\mu)=\int\mu_p(\U)d\nu(p)
\]
for all Borel $\U\subseteq\cs$.  If $\xi$ is the resulting barycenter of $P_\nu$, it follows from Theorem \ref{thm-barycenter} that $\xi$ is the mixture of $\nu$-random Bernoulli measures:

\begin{corollary}\label{cor-barycenter}
Let $\nu$ be a computable measure on $[0,1]$.  Then there is a computable measure $\xi$ on $\cs$ such that
\[
    \MLR_{\xi}=\bigcup_{p\in\MLR_\nu}\MLR_{\mu_p}.
\]
\end{corollary}

This result in implicit in \cite{Hoy13}  and was shown independently by V'yugin in \cite{Vyu12}.  

\subsection{Characterizing proper Bernoulli parameters}
An immediate consequence of Corollary \ref{cor-barycenter} is the following:

\begin{corollary}\label{cor-prop}
  Let $p\in[0,1]$ be proper.  Then there is some computable
  $\xi\in\mathscr{P}(\cs)$ such that $\MLR_{\mu_p}\subseteq \MLR_\xi$.
\end{corollary}

%
%
%
%

We thus have the consequence that every sequence that is Bernoulli $p$-random for some proper $p$ is itself proper.  This raises a natural question:  If $x$ is proper and Bernoulli $p$-random for some $p\in[0,1]$, must $p$ be proper?  We answer this question in the affirmative.  To do so, we will draw upon facts from the proof of the following result, discussed in the introduction, which is due to Kjos-Hanssen:

\begin{theorem}[\cite{Kjo10}]\label{thm-kj1}
For $p\in[0,1]$, if $x\in\MLR_{\mu_p}$, then $x\geq_T p$.
\end{theorem}

\begin{theorem}\label{thm-prop}
If $x$ is proper and $x\in\MLR_{\mu_p}$ for some $p\in[0,1]$, then $p$ is proper.
\end{theorem}
 
\begin{proof}

Fix $p$ such that $x\in\MLR_{\mu_p}$ and
  suppose that $x\in\MLR_\nu$ for some computable measure $\nu$.  By the proof of Theorem \ref{thm-kj1}
 there is a blind Martin-L\"of test
  $(\mathcal{V}_d)_{d\in\omega}$ such that (i) $(\mathcal{V}_d)_{d\in\omega}$ is a $\mu_q$-Martin-L\"of
    test for all $q\in[0,1]$, and (ii) for each $d\in\omega$, there is a Turing functional
    $\Theta_d$ such that $y\notin\mathcal{V}_d$ implies that
    \[
    \Theta_d(y)=\lim_{n\rightarrow\infty}\frac{\#\{i<n\colon
      y(i)=1\}}{n}.
    \]

  Since $x\in\MLR_{\mu_p}$, there is some $d\in\omega$ such that
  $x\notin\mathcal{V}_d$.  We define a total Turing functional
  $\Gamma$ in terms of the $\Pi^0_1$ class $\mathcal{P}_d:=\cs\setminus\mathcal{V}_d$ as follows:  For $y\in\cs$, to compute $\Gamma(y)$, we attempt to calculate
  $\Theta_d(y)$.  If $y\in\cs\setminus\mathcal{V}_d$, then
  $\Theta_d(y)$ will compute the relative frequency of 1s in $y$.
  However, if $y\notin\mathcal{P}_d$, there will be some stage $s$ such that $y\notin\mathcal{P}_{d,s}$ where
$(\mathcal{P}_{d,s})_{s\in\omega}$ is an effective sequence of clopen subsets of $\cs$ such that  $\bigcap_{s\in\omega}\mathcal{P}_{d,s}=\mathcal{P}_d$.
In this case, we will terminate the calculation of $\Theta_d(y)$ and $\Gamma(y)$ will switch to outputting 0s thereafter (following any initial bits that might have
  been calculated before the first stage $s$ at which we see
  $y\notin\mathcal{P}_{d,s}$).  Clearly, $\Gamma$ is total.
  
    Next, we set $\zeta=\nu\circ\Gamma^{-1}$, which
  is a computable measure because $\nu$ is computable and $\Gamma$ is
  a total Turing functional.  To see that $p$ is proper, note that since
  $x\in\mathcal{P}_d\cap\MLR_\nu$, $\Gamma(x)=p$, and $\Phi$ is defined on a set of $\nu$-measure one (as $\Phi$ is total), it
  follows by the preservation of randomness (Theorem \ref{thm-pres}(i)) that
  $\Gamma(x)=p\in\MLR_\zeta$.
\end{proof}

Combining Corollary \ref{cor-prop} and Theorem \ref{thm-prop}, we have shown:

\begin{corollary}
  For $x\in\MLR_{\mu_p}$ for some $p\in[0,1]$, $x$ is proper if and
  only if $p$ is proper. \label{thm:xpropiffpprop}
\end{corollary}

The above result shows that any value $p$ for which there is a proper $p$-random sequence must itself be proper.  Can this be shown to hold uniformly?  That is, given a computable measure $\nu$ on $\cs$, can we find a \emph{single} computable measure $\zeta$ such that 
\[
(\forall p\in[0,1])\bigl[\MLR_\nu\cap\MLR_{\mu_p}\neq\emptyset\;\Rightarrow\; p\in\MLR_\zeta\bigr]?
\] We answer this question in the affirmative.

\begin{theorem}\label{thm-uniform}
  If $\nu$ is a computable measure on $\cs$, then there is a computable measure $\zeta$ on $[0,1]$ such that for
  every $p\in[0,1]$ such that 
  $\MLR_\nu\cap\MLR_{\mu_p}\neq\emptyset$, we have
  $p\in\MLR_\zeta$.\label{thm:compatibility}
\end{theorem}

\begin{proof}
As in the proof of Theorem~\ref{thm:xpropiffpprop}, for every $p\in[0,1]$ and every $x\in\MLR_\nu\cap\MLR_{\mu_p}$, there is some $d\in\omega$ such that $x\in \mathcal{P}_d$ and thus a corresponding measure $\zeta_d$ (defined in terms of the functional $\Theta_d$) such that $p\in\MLR_{\zeta_d}$.  Moreover, the collection of measures $(\zeta_j)_{j\in\omega}$ is uniformly computable.  Thus we can define the convex combination $\zeta \eqdef \sum 2^{-(i+1)}\zeta_{i}$, which is itself a computable measure.  We claim that $\MLR_{\zeta_d}\subseteq\MLR_\zeta$.  Given $q\notin\MLR_\zeta$, there is some $\zeta$-Martin-L\"of test $(\U_i)_{i\in\omega}$ such that $q\in\bigcap_{i\in\omega}\U_i$.  Then since for each $i\in\omega$ we have
\[
2^{-(d+1)}\zeta_d(\U_i)\leq\sum_{i\in\omega}2^{-(j+1)}\zeta_j(\U_i)=\zeta(\U_i)\leq 2^{-i},
\]
 it follows that $(\U_i)_{i\geq d+1}$ is a $\zeta_d$-Martin-L\"of test containing $q$.  $\zeta$ is thus the desired measure.
\end{proof}

Let us say that a sequence is \emph{continuously proper} if it is random with respect to a continuous, computable measure, where a measure $\mu$ is continuous if $\mu(\{x\})=0$ for every $x\in\cs$. We have seen that a Bernoulli $p$-random sequence is proper if and only if $p$ is sequence.  Does a similar result hold for continuously proper $p$-random sequences?  One direction is straightforward:

\begin{proposition}\label{prop-contprop}
Let $p\in[0,1]$.  If $x\in\MLR_{\mu_p}$ is proper, then $x$ is continuously proper.
\end{proposition}

\begin{proof}
Since $x$ is $p$-random and proper, by Theorem \ref{thm-prop}, $p$ is random with respect to some computable measure $\eta$ on $[0,1]$.  If $\xi$ is the mixture of all Bernoulli measures with parameters that are random with respect to $\eta$, then we have, for every $z\in\cs$,
\[
\xi(\{z\})=\lim_{n\rightarrow\infty}\xi(z\uh n)=\lim_{n\rightarrow\infty}\int\mu_p(z\uh n)d\eta(p)= \int\lim_{n\rightarrow\infty}\mu_p(z\uh n)d\eta(p)=0
\]
(where the second equality follows from the dominated convergence theorem). Thus, $\xi$ is continuous and $x$ is continuously proper.
\end{proof}

We now show that the converse of Proposition \ref{prop-contprop} does not hold.

\begin{theorem}
  There is some proper $r\in[0,1]$ and $x\in\MLR_{\mu_r}$ such that
  $x$ is continuously proper but $r$ is not continuously proper.
\end{theorem}

\begin{proof}  By \cite[Theorem 3.2]{Por15}, there is a computable measure $\nu$ on $[0,1]$ with the following properties:
\begin{itemize}
\item[(i)] $\nu$ has countable support, i.e., there is a countable collection $\mathcal{C}$ of sequences such that $\nu(\mathcal{C})=1$;
\item[(ii)] every $y\in\mathcal{C}$ is an atom of $\nu$, i.e. $\nu(\{y\})>0$, and hence is computable  (as shown by Kautz \cite{Kau91} every atom of a computable measure is computable); and
\item[(iii)] there is one non-computable sequence $r$ such that $\MLR_\nu=\mathcal{C}\cup\{r\}$.  
\end{itemize}
 Given $x\in \MLR_{\mu_{r}}$, by Proposition \ref{prop-contprop}, $x$ is continuously proper. However, $r$ is not continuously proper:  Suppose otherwise.  Then there would be a computable, continuous measure $\zeta$ on $[0,1]$ such that  $r \in \MLR_{\zeta}$.  Then if $(\U_i)_{i\in\omega}$ is a universal $\zeta$-Martin-L\"of test, then $r\in\cs\setminus\U_i$ for some $i\in\omega$.  Since $\cs\setminus\U_i$ is a $\Pi^0_1$ class consisting of sequences that are random with respect to $\zeta$, and since $\zeta$ is continuous, there is no computable sequence in $\cs\setminus\U_i$.  Thus, $\nu(\cs\setminus\U_i)=0$, that $r\notin\MLR_{\nu}$ (since no $\nu$-random sequence can be contained in a $\nu$-null $\Pi^0_1$ class), a contradiction.
  \end{proof}

%
%


We conclude this section with the observation that Theorem \ref{thm-uniform} gives us a method for showing that certain mixtures of Bernoulli measures cannot be obtained effectively.  For example:

\begin{theorem}
  There is no computable mixture of Bernoulli measures that includes
  all computable Bernoulli measures.
\end{theorem}

\begin{proof}
Suppose there is some computable measure $\xi$ such that for all computable $c\in[0,1]$, there is some $x\in\MLR_\xi\cap\MLR_{\mu_c}$.  By Theorem \ref{thm-uniform},  there is a computable measure $\zeta$ on $[0,1]$ such that for
  every $p\in[0,1]$ such that 
  $\MLR_\nu\cap\MLR_{\mu_p}\neq\emptyset$, we have $p\in\MLR_\zeta$.  But then there is a computable measure $\zeta$ such that every computable sequence is random with respect to $\zeta$, which is impossible (one can always construct a computable sequence $z$ such that $\zeta(\{z\})=0$ by following the least non-ascending path of the $\zeta$-measure of members of $\str$).

\end{proof}

\section{The Turing Degrees of Bernoulli Random Sequences}\label{sec-4}

We now consider the Turing degrees of Bernoulli random sequences.  Levin \cite{ZvoLev70} and Kautz \cite{Kau91} independently proved that for every pair of computable measures $\mu$ and $\nu$, there is a $\mu$-almost total functional $\Phi_{\mu\rightarrow\nu}$ such that $\Phi_{\mu\rightarrow\nu}(\MLR_\mu)=\MLR_\nu$.  If, in addition, $\mu$ and $\nu$ are continuous and positive (i.e. $\nu(\sigma)>0$ for all $\sigma\in\str$), then $(\Phi_{\mu\rightarrow\nu})^{-1}=\Phi_{\nu\rightarrow\mu}$.  
 In fact, these result also holds for non-computable measures, at least as long as the two functionals have access to some encoding of $\mu$ and $\nu$ as binary sequences. In our context, to convert between randomness with respect to two Bernoulli measures $\mu_p$ and $\mu_q$ only requires oracle access to the parameters $p$ and $q$.  Thus we have:

\begin{theorem}[Levin \cite{ZvoLev70} / Kautz \cite{Kau91} ]\label{thm-lk}
 Let $p,q\in[0,1]$.
\begin{enumerate}
\item There is a $(p\oplus q)$-computable functional $\Phi_{p\rightarrow q}$ such that $\Phi_{p\rightarrow q}(\MLR_{\mu_p}^{p\oplus q})=\MLR_{\mu_q}^{p\oplus q}$.
\item $(\Phi_{p\rightarrow q})^{-1}=\Phi_{q \rightarrow p}$.

\end{enumerate}
\end{theorem}

Here we will focus primarily on the conversion between Lebesgue randomness and $p$-randomness for some $p\in[0,1]$.  In particular, the following is an immediate consequence of Theorem \ref{thm-lk}:

\begin{theorem}\label{thm-bernlk}
Let $p\in[0,1]$.
\begin{itemize}
\item[(i)] For every $x\in\MLR_{\mu_p}$ there is some $y\in\MLR^p$ such that
$x\equiv_T y\oplus p$.
\item[(ii)] For every $y\in\MLR^p$ there is some $x\in\MLR_{\mu_p}$ such that $x\equiv_T y\oplus p$.
\end{itemize}
\end{theorem}

\begin{proof}
(i) Let $x\in\MLR_{\mu_p}$.  By Theorem \ref{thm-lk}(1) applied to $q=\frac{1}{2}$, $\Phi_{p\rightarrow q}$ is a $p$-computable functional such that $\Phi_{p\rightarrow q}(x)\in\MLR^p$ since $x\in\MLR_{\mu_p}^{p\oplus q}$; set $y=\Phi_{p\rightarrow q}(x)$.  By Theorem \ref{thm-lk}(2), $\Phi_{q\rightarrow p}(y)=x$, where $\Phi_{q\rightarrow p}$ is $p$-computable.  Thus, we have $x\oplus p\equiv_Ty\oplus p$.   By Theorem \ref{thm-kj1}, since $x\in\MLR_{\mu_p}$, we have $x\geq_T p$, and thus the conclusion follows.

(ii) The argument is nearly identical as the above, except that we first apply the functional $\Phi_{q\rightarrow p}$ to $y\in\MLR^p$ with $q=\frac{1}{2}$ and argue as above.
\end{proof}

\subsection{Comparing degrees of Bernoulli parameters}
We first consider how the Turing comparability of two Bernoulli parameters relates to the Turing comparability of the associated Bernoulli random sequences.

\begin{theorem}\label{thm-berndeg1}
Every $x\in\MLR_{\mu_p}$ computes some $y\in\MLR_{\mu_q}$ if and only if $q\leq_Tp$.
\end{theorem}

\begin{proof}
($\Rightarrow$) For this direction, we will use Stillwell's generalization of Sack's Theorem \cite{Sti72}:  If $\lambda(\{x\colon a\leq_T b\oplus x\})>0$, then $a\leq_T b$ (see \cite[Theorem 8.12.6]{DowHir10}).  Now suppose that each $p$-random sequence computes a $q$-random sequence.  Given any $z\in\MLR^p$, by Theorem \ref{thm-bernlk}(ii), there is some $x\in\MLR_{\mu_p}$ such that $z\oplus p\equiv_T x$.  By assumption, $x$ computes some $y\in\MLR_{\mu_q}$, which in turn computes $q$ by Theorem \ref{thm-kj1}.  Thus we have $z\oplus p\geq_Tq$ for every $z\in\MLR^p$.  It follows that $\lambda(\{z\colon q\leq_T z\oplus p\})>0$, and hence by Stillwell's theorem, $q\leq_Tp$.

($\Leftarrow$) Suppose that $q\leq_Tp$.  Given $x\in\MLR_{\mu_p}$, by Theorem \ref{thm-bernlk}(i) there is some $z\in\MLR^p$ such that $z\oplus p\equiv_Tx$.   Since $q\leq_Tp$, $\MLR^p\subseteq \MLR^q$ and hence $z\in\MLR^q$.  Then by  Theorem \ref{thm-bernlk}(ii), there is some $y\in\MLR_{\mu_q}$ such that $y\equiv_Tz\oplus q$, and hence we have 
\[
y\equiv_Tz\oplus q\leq_T z\oplus p\equiv x.\qedhere
\]
\end{proof}

Clearly, if every $p$-random sequence is Turing-incomparable with every $q$-random sequence, then $p$ and $q$ are Turing-incomparable.  Interestingly, the converse does not hold:

\begin{proposition}
  For every pair of relatively random $p,q\in[0,1]$, there are
  sequences $x\in\MLR_{\mu_p}$ and $y\in\MLR_{\mu_q}$ such that
  $x\equiv_T y$.
\end{proposition}

\begin{proof}  Let $p,q\in[0,1]$ satisfy $p\in\MLR^q$ and $q\in\MLR^p$.  Since $p\in\MLR^q$, $\Phi_{\lambda\rightarrow\mu_q}(p)\in\MLR_{\mu_q}$ by Theorem \ref{thm-lk}.  Setting $y=\Phi_{\lambda\rightarrow\mu_q}(p)$, we have $y\equiv_{T}p\oplus q$ by Theorem \ref{thm-bernlk}(ii).  Similarly, since $q\in\MLR^p$, $\Phi_{\lambda\rightarrow\mu_p}(q)\in\MLR_{\mu_p}$ by Theorem \ref{thm-lk}.  Setting $x=\Phi_{\lambda\rightarrow\mu_p}(q)$, again we have $x\equiv_{T}p\oplus q$ by Theorem \ref{thm-bernlk}(ii).  Thus $x\equiv_Tp\oplus q\equiv_T y$.
\end{proof}

\bigskip

\subsection{On Turing degrees containing Bernoulli random sequences}
We now turn to the task of determining which Turing degrees contain a Bernoulli random sequence.  First, note that every Bernoulli random sequence computes a Martin-L\"of random sequence via von Neumann's trick, which is given by the following procedure:  Given a sequence $x$ as input, our procedure reads a pair of bits.  If it reads the pair 01, it outputs a 0, while if it read the pair 10, it outputs a 1; if the procedure reads either 00 or 11, it moves on to the next pair of input bits with no output.  One can verify that this procedure induces the Lebesgue measure, and hence, given any Bernoulli random as input, we get a Martin-L\"of random sequence as the output by the preservation of randomness.

Recall that a sequence $x$ has diagonally noncomputable (DNC) degree if there is some $f\leq_T a$ such that $f(n)\neq\phi_n(n)$ for every $n$ (where $(\phi_i)_{i\in\omega}$ is the standard enumeration of all partial computable functions). As every sequence that computes a subset of a Martin-L\"of random sequence has DNC degree (which can be deduced from results in \cite{Kjo09} and \cite{GreMil11}) we can conclude:

\begin{proposition}
Every Bernoulli random sequence has DNC degree.
\end{proposition}

Having DNC degree, however, does not characterize the degrees of Bernoulli random sequences.  

\begin{proposition}
There is a DNC degree that does not contain a Bernoulli random sequence.
\end{proposition}

\begin{proof}
As shown by Kumabe and Lewis \cite{KumLew09}, there is a sequence $x$ of DNC degree and minimal degree (if $y\leq_T x$, then either $y$ is computable or $y\equiv_Tx$).  However, no Bernoulli random sequence has minimal degree, since, as noted above, every Bernoulli random sequence computes a Martin-L\"of random sequence.
\end{proof}

Another candidate for characterizing the degrees of Bernoulli random degrees is the collection of Martin-L\"of random degrees, i.e., those Turing degrees that contain a Martin-L\"of random sequence.  As we now show, this characterization does not hold, as there are Bernoulli random degrees that contain no Martin-L\"of random sequence.  Before we prove this result, we first show that for every $p\in[0,1]$, there is some Bernoulli $p$-random that has Martin-L\"of random degree.

\begin{theorem}
For every $p\in[0,1]$, there is some $p$-random sequence $x$ and a Martin-L\"of random sequence $y$ such that $x\equiv_T y$.   
\end{theorem}

\begin{proof}
For $p\in[0,1]$, consider $\Omega^p=\sum_{U^p(\sigma)\downarrow}2^{-|\sigma|}$, Chaitin's $\Omega$ relative to the oracle $p$, where $U$ is a universal, prefix-free oracle machine.  As shown by Downey, Hirschfeldt, Miller, and Nies \cite{DowHirMil05}, $\Omega^p\in\MLR^p$ and
$\Omega^p\oplus p\equiv_T p'$.  By Theorem \ref{thm-bernlk}(ii), since $\Omega^p\in\MLR^p$, there is some $p$-random sequence $x\equiv_T \Omega_p\oplus p$.  In addition, by the Ku\v cera-G\'acs theorem (\cite{Kuc85},\cite{Gac86}), for every $a\geq_T\emptyset'$, there is some $y\in\MLR$ such that $y\equiv_T a$.  Thus there is some $y\in\MLR$ satisfying $y\equiv_Tp'$, which yields
\[
x\equiv_T (\Omega^p\oplus p)\equiv_T p'\equiv_T y.\qedhere
\]
\end{proof}

Thus, there is no $p\in[0,1]$ with the property that no $p$-random sequence has Martin-L\"of random degree.  However, we have the following:

\begin{theorem}\label{thm-berntd}
There is some $p\in[0,1]$ and some $y\in\MLR_{\mu_p}$ such that there is no $x\in\MLR$ satisfying $x\equiv_Ty$.
\end{theorem}

In the proof of Theorem \ref{thm-berntd}, we will make use of several results. First we have two relativizations of standard results in the theory of algorithmic randomness (the proofs of which are direct relativizations of the proofs of the original theorems).  First, we have the relative version of what is sometimes referred to as the \emph{XYZ Theorem}, originally due to Miller and Yu \cite{MilYu08}.

\begin{theorem}
For $x,y,z,p\in\cs$, if $x\in\MLR^p$ and $x\leq_Ty\oplus p$ for $y\in\MLR^{z\oplus p}$, then $x\in\MLR^{z\oplus p}$.
\end{theorem}

Next, we have the relative version of van Lambalgen's theorem \cite{Van90}:

\begin{theorem}
For $x,y,p\in\cs$,  $x\oplus y\in\MLR^p$ if and only if $x\in\MLR^{y\oplus p}$ and $y\in\MLR^p$.
\end{theorem}

Lastly, we will draw upon a recent result involving $K$-trivial sequences.  Recall that a sequence $a\in\cs$ is $K$-trivial if there is some $c\in\omega$ such that 
\[
K(a\uh n)\leq K(n)+c
\]
for all $n\in\omega$ (here $K(\sigma)=\min\{|\tau|\colon U(\tau)=\sigma\}$ is the prefix-free Kolmogorov complexity of $\sigma\in\str$, where $U$ is a universal prefix-free machine; see \cite[Chapter 5]{Nie09} or \cite[Chapter 11]{DowHir10} for more details).  As shown by Nies \cite{Nie05}, $a$ is $K$-trivial if and only if $\MLR^a=\MLR$ (a property referred to as being \emph{low for Martin-L\"of randomness}.

To establish our result, we make use of the following result about $K$-trivial sequences due to Bienvenu, Greenberg, Ku\v cera, Nies, and Turetsky \cite{BieGreKuc16}:

\begin{theorem}\label{thm-ktriv}
There is a $K$-trivial sequence $p$ such that for every Martin-L\"of random $x=x_0\oplus x_1$,  $p\not\leq_Tx_i$ for some $i\in\{0,1\}$.
\end{theorem}

\begin{proof}[Proof of Theorem \ref{thm-berntd}] Fix a $K$-trivial $p$ as in the proof of Theorem \ref{thm-ktriv}.  We claim that $p$ is the desired Bernoulli parameter.  Let $a\oplus b\in\MLR$.  By Theorem \ref{thm-ktriv}, without loss of generality we can assume that $p\not\leq_Ta$.  Splitting $a$ into $a_0\oplus a_1$, we have $p\not\leq_Ta_i$ for $i\in\{0,1\}$.  Moreover, for each $i\in\{0,1\}$, by the relative version of van Lambalgen's theorem and the fact that  $a_0\oplus a_1\in\MLR^p$, we have $a_i\in\MLR^{a_{1-i}\oplus p}$.

Next, we apply the Levin-Kautz theorem:  For each $i\in\{0,1\}$, since $a_i\in\MLR^p$,  by Theorem \ref{thm-bernlk}(ii) there is some $y_i\in\MLR_{\mu_p}$ such that $y_i\equiv_T a_i\oplus p$.  Suppose now for the sake of contradiction that for each $i\in\{0,1\}$, there is some $x_i\in\MLR$ such that $x_i\equiv_T y_i$.  By the relativized version of the XYZ Theorem, since $x_i\in\MLR^p$, $x_i\leq_T a_i\oplus p$, and $a_i\in\MLR^{a_{1-i}\oplus p}$, it follows that $x_i\in\MLR^{a_{1-i}\oplus p}=\MLR^{x_{1-i}}$.  We can thus conclude that $x=x_0\oplus x_1\in\MLR$ by van Lambalgen's theorem.  However, since $p\leq_Ty_i\equiv_T x_i$ for $i\in\{0,1\}$, this contradicts our assumption that $P$ cannot be computed by two halves of a random sequence.  Thus, either $\deg_T(y_0)$ or $\deg_T(y_1)$ is a Bernoulli random Turing degree that contains no Martin-L\"of random sequence.
\end{proof}

The proof of Theorem \ref{thm-berntd} yields many Bernoulli degrees that are not Martin-L\"of random degrees:

\begin{corollary}
There are uncountably many Turing degrees that contain a Bernoulli random sequence and no Martin-L\"of random sequence.
\end{corollary}

\begin{proof}
For every Martin-L\"of random sequence $x$, if we split $x$ into $(x_0\oplus x_1)\oplus(x_2\oplus x_3)$, then at least one of $x_i\oplus p$ (where $p$ is as in the proof of Theorem \ref{thm-berntd}) for $i\in\{0,1,2,3\}$ is Turing equivalent to a $p$-random sequence but not Turing equivalent to a Martin-L\"of random sequence.
\end{proof}

Note that any parameter $p$ for which there is a $p$-random sequence not Turing equivalent to a Martin-L\"of random sequence must of necessity be non-proper.  For given a proper sequence $p$, any $p$-random sequence is proper by Corollary \ref{cor-prop} and hence is Turing equivalent to a Martin-L\"of random sequence (by the Levin-Kautz theorem and randomness preservation).  Characterizing precisely which parameters $p$ yield Theorem \ref{thm-berntd} remains open.

\begin{question}
For which $p\in[0,1]$ is there a Bernoulli $p$-random sequence that is not Turing equivalent to any Martin-L\"of random sequence?
\end{question}

\section{Blind randomness with respect to a Bernoulli measure}\label{sec-5}

In addition to showing Theorem \ref{thm-kj1}, another contribution of \cite{Kjo10} is the introduction of blind randomness and its application to the study of randomness with respect to a Bernoulli measure.
For  a non-computable measure  $\mu$ on $\cs$, a sequence of uniformly $\Sigma^0_1$ classes $(\U_i)_{i\in\omega}$ is called a \emph{blind} $\mu$-Martin-L\"of test if $\mu(\U_i)\leq 2^{-i}$ for every $i\in\omega$; that is, the test does not make use of the measure as an oracle, unlike the standard definition of Martin-L\"of randomness with respect to non-computable measures.  Moreover, we say that $x\in\cs$ is \emph{blind $\mu$-Martin-L\"of random} if $x$ passes every blind $\mu$-Martin-L\"of test.  

In the context of Bernoulli measures, for $p\in[0,1]$, a blind $\mu_p$-Martin-L\"of random test is one that does not use the parameter $p$ as an oracle.  Kjos-Hanssen proved that, for the purposes of defining Bernoulli randomness, having oracle access to the Bernoulli parameter is optional:

\begin{theorem}[Kjos-Hanssen \cite{Kjo10}]\label{thm-kh}
For $p\in[0,1]$, $x\in\cs$ is $\mu_p$-Martin-L\"of random if and only if $x$ is blind $\mu_p$-Martin-L\"of random.
\end{theorem}

To date, no initial segment complexity characterization of blind randomness has been provided.  In this section, we show that several reasonable approaches to characterizing blind randomness in terms of initial segment complexity fail to do so.
 Let us review several relevant definitions.  
 
 Recall that a \emph{semimeasure} $\rho:\str\rightarrow[0,1]$ is a function satisfying $\rho(\emptyset)\leq 1$ and $\rho(\sigma)\geq\rho(\sigma0)+\rho(\sigma1)$ for every $\sigma\in\str$.  Moreover, a semimeasure $\rho$ is \emph{left-c.e.}\ if, uniformly in $\sigma$, each value $\rho(\sigma)$ is left-c.e., i.e., the limit of a computable, non-decreasing sequence of rational numbers.  Levin \cite{ZvoLev70} proved the existence of a universal, left-c.e.\ semimeasure $M$:  for every left-c.e.\ semimeasure $\rho$, there is some $c\in\omega$ such that $M\geq c\cdot\rho$.  We then define the \emph{a priori complexity} of $\sigma$ to be $\KA(\sigma):=-\log M(\sigma)$.   These notions are straightforwardly relativizable to any $z\in\cs$.

For $p\in[0,1]$, it is a standard result that a sequence $x\in\cs$ is Bernoulli $p$-random if and only if 
$\KA^p(x\uh n)\geq-\log\mu_p(x\uh n)-O(1)$ (see \cite[Proposition 2.2]{MilRut18} for a proof of the more general statement that holds for all noncomputable measures).  Here we consider those sequences that satisfy
\[
\KA(x\uh n)\geq-\log\mu_p(x\uh n)-O(1),
\]
where we remove the oracle $p$ and consider unrelativized a priori complexity.
  As we now show, this notion is not, in general, sufficient for Bernoulli $p$-randomness.

\begin{theorem}\label{thm-noblind}
  Let $p\in[0,1]$ satisfy $p\geq_T\emptyset'$.  Then there is some
  $y\notin\MLR_{\mu_p}$ such that
  \[
  \KA(y\uh n)\geq-\log\mu_p(y\uh n)-O(1).
  \]
\end{theorem}

\begin{proof}
Fix $c\in\omega$ and consider the $\Pi^{0,p}_1$ class 
\[
\mathcal{P}=\{x\in\cs\colon (\forall n)[\KA(x\uh n)\geq-\log\mu_p(x\uh n)-c]\}.
\]
Observe that for  $T=\{\sigma\in\str\colon \colon \KA(\sigma)\geq -\log\mu_p(\sigma)-c\}$, we have $T\leq_T p$; indeed, the predicate $ \KA(\sigma)\geq -\log\mu_p(\sigma)-c$ is computable in $p\oplus \emptyset'\equiv_T p$.  We claim that $T=\{\sigma\in\str\colon (\exists x\in\P)[\sigma\prec x]\}$, the set of extendible  nodes of $\P$. In particular, we show that if $\sigma\in T$, then either $\sigma0$ or $\sigma1$ is in $T$.  

Suppose that for some $\sigma\in T$, $\sigma0\notin T$ and $\sigma1\notin T$.  Since $\sigma\in T$, we have $\KA(\sigma)\geq-\log\mu_p(\sigma)-c$, which implies that 
\begin{equation}\label{eq1}
M(\sigma)\leq 2^c\mu_p(\sigma).
\end{equation}
  For $i\in\{0,1\}$, $\sigma i\notin T$ implies that $\KA(\sigma i)<-\log\mu_p(\sigma i)-c$, and hence we have
\begin{equation}\label{eq2}
M(\sigma i)> 2^c\mu_p(\sigma i)
\end{equation}
for $i\in\{0,1\}$.  Combining Equations (\ref{eq1}) and (\ref{eq2}) and using the properties of measures and semimeasures, we have
\[
2^c\mu_p(\sigma)\geq M(\sigma)\geq M(\sigma 0)+M(\sigma1)>2^c(\mu_p(\sigma0)+\mu_p(\sigma1))=2^c\mu_p(\sigma),
\]
which is impossible.  Thus, either $\sigma0\in T$ or $\sigma1\in T$.  Since every $\sigma\in T$ has an extension in $T$, every $\sigma$ extends to an infinite path through $T$.  Thus the claim follows.

%
As $p\geq_T T=\{\sigma\in\str\colon (\exists x\in\P)[\sigma\prec x]\}$, it follows that $p$ can compute the leftmost path through $\P$.  That is, there is some $y\in \P$ such that (i) $\KA(y\uh n)\geq-\log\mu_p(y\uh n)-c$ for all $n$ and (ii) $y\leq_Tp$, from which it follows that $y\notin\MLR_{\mu_p}$.
\end{proof}

Note that since $\KA(\sigma)\leq K(\sigma)+O(1)$ for every $\sigma\in\str$, we can also conclude that the condition that $K(x\uh n)\geq-\log\mu_p(x\uh n)-O(1)$ does not imply Bernoulli $p$-randomess.

\begin{corollary}
  Let $p\in[0,1]$ satisfy $p\geq_T\emptyset'$.  Then there is some
  $y\notin\MLR_{\mu_p}$ such that
  \[
  K(y\uh n)\geq-\log\mu_p(y\uh n)-O(1).
  \]
\end{corollary}

A third attempt at characterizing  blind $p$-randomness can be given in terms of supermartingales.  Recall that, for a computable measure $\mu$ on $\cs$, a $\mu$-martingale is a computable function $M:\str\rightarrow\mathbb{R}^{\geq 0}$ that satisfies 
\[
M(\sigma)=\dfrac{\mu(\sigma0)}{\mu(\sigma)}M(\sigma0)+\dfrac{\mu(\sigma1)}{\mu(\sigma)}M(\sigma1)
\]
for every $\sigma\in\str$.  It is not hard to show that for every computable $\mu$-martingale $M$, there is some computable measure $\nu$ such that $M=\frac{\nu}{\mu}$.  

Similarly, a c.e.\ $\mu$-supermartingale is a left-c.e.\ function $M:\str\rightarrow\mathbb{R}^{\geq 0}$ that satisfies 
\[
M(\sigma)\geq\dfrac{\mu(\sigma0)}{\mu(\sigma)}M(\sigma0)+\dfrac{\mu(\sigma1)}{\mu(\sigma)}M(\sigma1)
\]
for every $\sigma\in\str$.  As with the case of computable $\mu$-martingales, one can show that for every c.e.\ $\mu$-supermartingale $M$, there is some left-c.e.\ semimeasure $\delta$ such that $M=\frac{\delta}{\mu}$.

Lastly, a $\mu$-martingale $M$ \emph{succeeds} on a sequence $x$ if $\limsup_{n\in\omega} M(x\uh n)=\infty$; one can define success for a supermartingale in the same way.

  In \cite{KjoTavTha14},  Kjos-Hanssen, Taveneaux, and Thapen studied blind $\mu_p$-martingales for $p\in[0,1]$ with the aim of study a blind analogue of computable randomness with respect to the measure $\mu_p$.  Although they did not phrase their definition in this way, one can straightforwardly show that, for each $p\in[0,1]$, a blind $p$-martingale as defined in \cite{KjoTavTha14} has the form $M=\frac{\nu}{\mu_p}$, where $\nu$ is a computable measure (and not a $p$-computable measure).  Similarly, we can define a blind $p$-supermartingale to be any function of the form $M=\frac{\delta}{\mu_p}$, where $\delta$ is a left-c.e.\ semimeasure.  
  
Kjos-Hanssen, Taveneaux, and Thapen showed that there is some for every Martin-L\"of random $p\in[0,1]$, there is a $p$-computable sequence $x$ such that no blind $p$-martingale succeeds on $x$ .  We show a similar but less general result:

\begin{theorem}
  Let $p\in[0,1]$ satisfy $p\geq_T\emptyset'$. Then there is some
  $y\notin\MLR_{\mu_p}$ such that no blind $p$-supermartingale
  succeeds on $y$.
\end{theorem}

\begin{proof}
  Let $y$ be as in the proof of Theorem \ref{thm-noblind}.  Suppose there is a blind
  $p$-supermartingale $M$ that succeeds on $y$.  That is, for every
  $d\in\omega$, there is some $n\in\omega$ such that $M(y\uh n)>2^d$.
  Since $M=\frac{\delta}{\mu_p}$ for some left-c.e.\ semimeasure
  $\delta$, we have that for every $c\in\omega$, there is some
  $n\in\omega$ such that
  \[
  \frac{\delta(y\uh n)}{\mu_p(y\uh n)}>2^d.
  \]
  Applying the negative logarithm to both sides yields
  \begin{equation}\label{eq3}
  -\log\delta(y\uh n)<-\log\mu_p(y\uh n)-d.
  \end{equation}
 By the optimality of the universal
  left-c.e.\ semimeasure, there is some $e\in\omega$ such that
  $2^e\cdot M\geq \delta$, or equivalently,
  $-\log M(\sigma)\leq -\log \delta(\sigma) + e$.  Combining this with
  Equation (\ref{eq3}) yields
  \[
  -\log M(y\uh n)<-\log\mu_p(y\uh n)-(d-e).
  \]
  Thus, for every sufficiently large $c\in\omega$, there is some
  $n\in\omega$ such that
  \[
  \KA(y\uh n)<-\log\mu_p(y\uh n)-c,
  \]
  which contradicts Theorem \ref{thm-noblind}
\end{proof}

Note that 	the choice of $p$ in Theorem \ref{thm-noblind} requires that $p\geq_T\emptyset'$.  We do not know whether we can drop this assumption.

\begin{question}
  For each noncomputable $p\in[0,1]$, is there some $y\notin\MLR_{\mu_p}$
  such that
  \[
  \KA(y\uh n)\geq-\log\mu_p(y\uh n)-O(1)?
  \]
\end{question}

\section*{Acknowledgement}

Many thanks to Quinn Culver for extremely helpful conversations during the early phases of this project.

\bibliographystyle{alpha} \bibliography{Bernoulli}

\end{document}